\documentclass[runningheads]{llncs}
\usepackage{enumerate}
\usepackage{graphicx}
\usepackage[]{color}
\usepackage{theorem}
\usepackage{calc}
\usepackage{ifthen}
\usepackage{epsfig}
\usepackage{graphicx}
\usepackage{latexsym}
\usepackage{amssymb,amsmath}
\usepackage[]{hyperref}

\newcommand{\red}{\color{red}}

\newcommand{\black}{\color{black}}

\newtheorem{facts}[example]{Facts}

\newtheorem{examples}[example]{Examples}

\begin{document}
\title{Harary polynomials}
%
%
\author{
O. Herscovici\inst{1}\thanks{Supported by the  Israeli Science Foundation (ISF) grant \#1144/16. }
\and
J.A. Makowsky\inst{2}
\and
V. Rakita\inst{3}\thanks{Supported by a grant from the Irwin and Joan Jacobs Graduate School 
of the Technion--Israel Institute of Technology}
}
\authorrunning{O. Herscovici et al.}
%
\institute{
Department of Mathematics, University of Haifa, Haifa, Israel
\email{orli.herscovici@gmail.com}
\and
Department of Computer Science, Technion-IIT, Haifa, Israel
\email{janos@cs.technion.ac.il}
\and
Department of Mathematics, Technion-IIT, Haifa, Israel
\email{vsevolod@campus.technion.ac.il}
}
\maketitle              
Last revised and expanded,
July 1, 2020
\begin{abstract}
Given a graph property $\mathcal{P}$,
F. Harary introduced in 1985 $\mathcal{P}$-colorings, graph colorings where each colorclass induces a graph in $\mathcal{P}$.
Let $\chi_{\mathcal{P}}(G;k)$ counts the number of $\mathcal{P}$-colorings of $G$ with at most $k$ colors.
It turns out that $\chi_{\mathcal{P}}(G;k)$ is a polynomial in $\mathbb{Z}[k]$ for each graph $G$.
Graph polynomials of this form are called Harary polynomials.
In this paper we investigate properties of Harary polynomials and compare them with properties
of the classical chromatic polynomial $\chi(G;k)$.
We show that the characteristic and the Laplacian polynomial, the matching, the independence and the domination polynomial
are not Harary polynomials.
We show that for various notions of sparse, non-trivial properties $\mathcal{P}$, 
the polynomial $\chi_{\mathcal{P}}(G;k)$ is, 
in contrast to $\chi(G;k)$,
not a chromatic, and even not an edge elimination invariant.  
Finally we study whether the Harary polynomials are definable in
Monadic Second Order Logic.


\end{abstract}
\newif\ifskip
\skiptrue
\newcommand{\stirlingii}{\genfrac{\{}{\}}{0pt}{}}
\newcommand{\stirlingi}{\genfrac{[}{]}{0pt}{}}
\newcommand{\ZZ}{\mathbb{Z}}
\newcommand{\NN}{\mathbb{N}}
\newcommand{\PP}{\mathbb{P}}
\newcommand{\RR}{\mathbb{R}}
\newcommand{\EE}{\mathbb{E}}
\newcommand{\CC}{\mathbb{C}}
\newcommand{\cC}{\mathcal{C}}
\newcommand{\cP}{\mathcal{P}}
\newcommand{\cD}{\mathcal{D}}
\newcommand{\cR}{\mathcal{R}}
\newcommand{\aA}{\mathcal{A}}
\newcommand{\bP}{\mathbf{P}}
\newcommand{\bNP}{\mathbf{NP}}
\newcommand{\goesto}{\rightarrow}
\newcommand{\isom}{\cong}
\newcommand{\SOL}{\mathrm{SOL}}
\newcommand{\MSOL}{\mathrm{MSOL}}
\newcommand{\FOL}{\mathrm{FOL}}

\section{Introduction and main results}
\label{se:intro}
\subsection{Prelude}
This paper initiates a systematic study of univariate graph polynomials, called {\em Harary polynomials},
or {\em generalized chromatic polynomials}.
We explore how the Harary polynomials differ from the traditional univariate graph polynomials from the
literature, among them the characteristic and Laplacian polynomial, the original chromatic polynomial, the matching polynomial,
the independence and the clique polynomial. 

The paper uses techniques developed in the last twenty years by the second author and his collaborators,
I. Averbouch, B. Godlin,  T. Kotek and E. Ravve, and shows that these techniques form a solid body of tools,
which can be applied to study Harary polynomials. The results show a coherent picture, even if no new techniques
are developed in this paper.

\subsection{Harary polynomials}
Let $\cP$ be a graph property.
In \cite{harary1985conditional} 
F. Harary introduced 
the notion of $\cP$-coloring as a generalization of proper colorings, which he called {\em conditional colorings}.
Let $G = <V(G),E(G)>$ be a graph and $[k] = \{1, 2, \ldots ,k \}$.
A function $c: V(G) \rightarrow [k]$  is a $\cP$-coloring with at most $k$ colors if for every $i \in [k]$ the
set $\{ v \in V(G): f^{-1}(i) \} $ induces a graph in $\cP$.
If $\cP$ is the property that $E(G)= \emptyset$, i.e., $\cP$ consists of all the edgeless graphs, this gives the proper colorings.
Other properties of $\cP$ studied in the literature are
$G$ is connected, $G$ is triangle-free or $G$ is a complete graph.
F. Harary introduced $\cP$-colorings with the idea that they might behave in a similar way
to proper colorings.
$\cP$-colorings were further studied in
\cite{brown1987generalized}.

Let $\chi_{\cP}(G;k)$ be the number of $\cP$-colorings of $G$ with at most $k$ colors,
and $\chi_{\cP}(G)$ be the {\em $\cP$-chromatic number}, which is the least $k$ such that $G$ has a $\cP$-coloring.

$\chi(G;k)$ is the {\em chromatic polynomial}, i.e., the Harary polynomial for $\cP$ containing all the edgeless graphs.
Generalizing Birkhoff's Theorem from 1912 for $\chi(G;k)$ 
it was noted in
\cite{makowsky2006polynomial}
that for every finite graph $G$ the counting function $\chi_{\cP}(G;k)$ is a polynomial in 
$\ZZ[k]$,
see also \cite{kotek2011counting}.
The family of poynomials $\chi_{\cP}(G;k)$ indexed by graphs $G$ is a graph polynomial called a {\em Harary polynomial}
in \cite{makowsky2019logician},
which can be written as
\begin{gather}
\chi_{\cP}(G; x) = \sum_{i \geq 1} b_i^{\cP}(G) x_{(i)},
\label{eq:harary}
\end{gather}
where $b_i^{\cP}(G)$ is the number of partitions of $V(G)$ into $i$ parts, where each part induces a graph in ${\cP}$,
and $x_{(i)}$ is the falling factorial.

\begin{facts}
\label{facts}
For every graph property $\cP$ and every graph $G$ of order $n$ we have
\begin{enumerate}[(i)]
\item
$b_0^{\cP}(G)=0$, if the nullgraph $(\emptyset, \emptyset)$ is not in $\cP$.
\item
$b_1^{\cP}(G) \in \{0,1\}$ and $b_1^{\cP}(G) =1$ iff $G \in {\cP}$.
\item
$b_n^{\cP}(G) \in \{0,1\}$ and $b_n^{\cP}(G) =1$. 
\item
The polynomial $\chi_{\cP}(G; k)$ is monic of degree $|V(G)|$ iff $K_1 \in {\cP}$.
\item
If $k < n$ and $\chi_{\cP}(G,k) =0$ then for all $0 <\ell < k$ also $\chi_{\cP}(G, \ell) =0$.
\end{enumerate}
\end{facts}

\begin{examples}
Here are some $\cP$-colorings and Harary polynomials studied in the literature.
\begin{enumerate}[(i)]
\item
$mcc_t$-colorings. Here $\cP = \cP_t$, where $\cP_t$ is the graph property such that 
the connected components of $H \in \cP_t$ have order at most $t$.
\\
For $t=1$ these are the proper colorings, for $t=2$ these are the $\cP_3$-free colorings.
They were introduced in \cite{linial2007graph} with a slightly different notation.
\item
Let $H$ be a connected graph of order $t$.
$DU(H)$ consist of non-empty disjoint unions of copies of $H$.
$DU(H)$-colorings are $mcc_t$-colorings.
They are studied in
\cite{goodall2018complexity}.
\item
Let $Fr(H)$ be the class of all graphs which do not contain $H$ as an induced subgraphs, which we call
$H$-free graphs.
$Fr(H)$-colorings are studied in \cite{achlioptas1997complexity,brown1996complexity,brown1987generalized}.
$K_2$-free colorings are just the proper colorings.
\item
A graph property $\cP$ is additive if it is closed under forming disjoint unions.
$\cP$ is hereditary, if it is closed under induced subgraphs.
A coloring is $\mathcal{AH}$ if it is a $\cP$-coloring for some $\cP$ which is both additive and hereditary.
$\mathcal{AH}$-colorings were studied in \cite{ar:Farrugia04}.
\item
The {\em adjoint polynomial $A(G;x) = \chi_{\cP}(G; k)$} is defined by taking  $\cP$ to be the 
class of all complete graphs.  It was introduced in \cite{liu1997new}, see also \cite{bencs2017one}.
\item
If $\cP$ consists of all connected graphs, we speak of
{\em convex colorings}, and put $C(G;x) = \chi_{\cP}(G; k)$, see 
\cite{goodall2018complexity,rota1964foundations,simon2011counting}.
\end{enumerate}
\end{examples}

The purpose of this paper is to initiate the study of Harary polynomials by comparing them to the chromatic polynomial.

\subsection{The chromatic polynomial and edge elimination invariants}
One of the fundamental properties of the chromatic polynomial is its characterization via edge elimination properties.
Given a graph $G$ and an edge $e \in E(G)$ we denote by 
$G_{-e}$,
$G_{/e}$ and
$G_{\dagger e}$
the graphs obtained from $G$ by deleting, contracting and extracting the edge $e$.
Extraction deletes $e=(u,v)$ together with the vertices $u,v$ and all the edges incident with $u$ or $v$.
A graph parameter $p(G)$ is an {\em edge elimination (EE) invariant}, see \cite{averbouch2010extension}, if it can be written
as a certain linear combination  of
$p(G_{-e})$,
$p(G_{/e})$,
$p(G_{\dagger e})$.

It is well known that $\chi(G;k)$ is an EE-invariant even without using $p(G_{\dagger e})$.
Other EE-invariants are the matching polynomials, some version of the Tutte polynomial and many others,
\cite{trinks2011covered,trinks2012proving}. However, the original Tutte polynomial is not an EE-invariant.
An alternative name for EE-invariants is DCE-invariants, for {\em Deletion, Contraction} and {\em Extraction}.

\begin{theorem}[\cite{averbouch2008most,averbouch2010extension}]
\label{th:xi-1}
There is a graph polynomial $\xi(G;x,y,z)$
\begin{gather}
\xi(G;x,y,z)=\sum_{A,B\subseteq E(G)} x^{c(A\cup B)-cov(B)}y^{|A|+|B|-cov(B)}z^{cov(B)}=
\notag \\
\sum_{A,B\subseteq E(G)} x^{c(A\cup B)}y^{|A|+|B|} (\frac{z}{xy})^{cov(B)}
\notag
\end{gather}
such that
\begin{enumerate}[(i)]
\item
$\xi(G;x,y,z)$ is an edge elimination invariant.
\item
$\xi(G;x,y,z)$ is universal, i.e., every other graph parameter $p(G)$ which is an edge elimination invariant
is a substitution instance of $\xi(G;x,y,z)$, i.e., it can be obtained from $\xi(G;x,y,z)$ by substituting
replacing $x,y,z$ by a polynomial in the indeterminates $x, x^{-1}, y, y^{-1}, z, z^{-1}$.
\end{enumerate}
Here
\begin{itemize}
\item
the summation is over $A,B\subseteq E(G)$ such that the vertex subsets $V(A), 
V(B)$ covered by $A$ and $B$, respectively, are disjoint, 
\item
$c(A)$ is the number of connected components in $(V(G),A)$, 
and 
\item
$cov(B)$ is the number of covered connected component of $B$, 
i.e. the number of connected components of $(V(B),B)$.
\end{itemize}
\end{theorem}

\subsection{$\MSOL$-definable graph polynomials}
The language of graphs has one binary relation symbol for the edge relation.
If we fix $k$ we note that $\chi(G;k) > 0$ iff $G$ is $k$-colorable.
This can be expressed by  formula in  monadic second order logic $\MSOL$ in the language of graphs by the formula
\begin{gather}
\exists V_1 
\exists V_2 
\ldots
\exists V_k 
(Partition(V_1, V_2, \ldots , V_k) \wedge \bigwedge_{i=1}^k Indep(V_i)).
\notag
\end{gather}
$Partition(V_1, V_2, \ldots , V_k)$ and $Indep(V_i)$ are first order expressible in the language of graphs.
The same works for Harary polynomials, provided $\cP$ is $\MSOL$-definable.
Checking whether a graph $G$ is $k$-colorable is $\bNP$-complete.
For the complexity of checking whether a graph is $\cP$-colorable for various graph properties $\cP$, 
the reader may consult
\cite{achlioptas1997complexity,brown1996complexity,goodall2018complexity}.

However, using Courcelle's celebrated Theorem,
\cite[Chapter 13]{downey2013fundamentals} and  \cite[Chapter 11]{flum2006parameterized},
$\MSOL$-definability implies that checking whether a graph $G$ is $k$-colorable is fixed parameter tractable (FPT)
for graphs of bounded tree-width, and even for graphs of bounded clique-width or rank-width.

For the chromatic polynomial one looks at the problem of computing the value of $\chi(G;k)$ for given
$k$ as a function of $G$. For $k= 0,1,2$ this is computable in polynomial time, whereas for $k \geq 3$
this is $\sharp \bP$-complete,
\cite{linial1986hard}. 
For graphs of fixed tree-width $w$, this is still in FPT. 
To see this one can use an extension of Courcelle's Theorem to the class of $\MSOL$-definable
graph polynomials,
\cite{courcelle2001fixed}.

The language of hypergraphs has two unary predicates $V$ and $E$ for vertices and edges which partition the universe,
and a binary incidence relation $R$ saying that vertices are connected by edges.
We denote by $\MSOL_g$ ($\MSOL_h$)  the monadic second order logic in the language of graphs (hypergraphs).

\begin{proposition}
Let $\cP$ be a graph property definable in $\MSOL_g$ ($\MSOL_h$).
Then checking whether a graph $G$ is $\cP$-colorable with $k$ colors is
definable in $\MSOL_g$ ($\MSOL_h$).
\end{proposition}

\begin{theorem}
[\cite{makowsky2014connection}]
\label{th:A}
$\chi(G;k)$ is not an $\MSOL_g$-definable polynomial,
but it is $\MSOL_h$-definable.
\end{theorem}
\begin{proof}
For fixed $k$ we write
$$
\exists U_1, \ldots, \exists U_k \bigwedge_{j \in [k]} \phi_{\cP}(U_j)
$$
where $U_j$ are sets of vertices and $\phi_{\cP}(U_j)$ says that $U_j$ induces a graph in $\cP$, provided $U_j$
is not empty.
\hfill $\Box$
\end{proof}

\begin{theorem}
\label{th:xi-2}
The most general EE-invariant $\xi(G;x,y,z)$ is $\MSOL_h$-definable for graphs with a linear order
on the vertices. Furthermore, this definition is invariant under the particular order of the vertices.
\end{theorem}
As there is no published proof of this, we include a proof here in the Appendix \ref{app:xidef}.
To prove that  $\chi(G;k)$ is not $\MSOL_g$-definable
we use the method of connection matrices, explained in Section \ref{se:connection}.
To prove that  $\chi(G;k)$ is $\MSOL_h$-definable
we use  that $\chi(G;k)$ is an EE-invariant and Theorems \ref{th:xi-1} and \ref{th:xi-2}.
We do not know a direct method, without the use of an $\MSOL_h$-definable EE-invariant, to show that $\chi(G;k)$ 
is indeed $\MSOL_h$-definable.

Theorem \ref{th:A}
still  implies that evaluating $\chi(G;k)$ is fixed parameter tractable (FPT) for graphs of tree-width at most $w$,
whereas for graphs of clique-width $w$
this is still open,
\cite{averbouch2011universal,makowsky2006computing}. 

\subsection{Main results}

A graph property $\cP$ is {\em trivial} if it is empty, finite (up to isomorphisms), or contains
all finite graphs.
Our main question in this paper asks whether Courcelle's Theorem and its variations
can be applied to Harary polynomials for non-trivial graph properties. 
This amounts to asking: 
\begin{enumerate}[(i)]
\item
Are there non-trivial graph properties $\cP$ such that the Harary polynomial $\chi_{\cP}(G,x)$ is  
$\MSOL_g$-definable? 
\item
Are there non-trivial graph properties $\cP$ such that the Harary polynomial $\chi_{\cP}(G,x)$
is an EE-invariant and hence $\MSOL_h$-definable?
\end{enumerate}

Recall that a graph property $\cP$ is {\em hereditary (monotone, minor-closed)}
if it is closed under taking induced subgraphs (subgraphs, minors).
Clearly, 
if $\cP$ is minor-closed, it is also monotone, and
if $\cP$ is monotone, it is also hereditary.

A graph property $\cP$ is {\em ultimately clique-free} if there exist $t \in \NN$ such that no graph $G \in \cP$
contains a $K_t$, i.e., a complete graph of order $t$.
Analogously, $\cP$ is {\em ultimately  biclique-free} if there exist $t \in \NN$ such that no graph $G \in \cP$
has $K_{t,t}$  as a subgraph (not necessarily induced). $K_{t,t}$  is the  complete bipartite graph of order $2t$.
Clearly, biclique-free implies clique-free, but not conversely.

\begin{theorem}
\label{th:main-1}
Let  $\cP$ be a graph property and $\chi_{\cP}(G,x)$ the Harary polynomial associated with  $\cP$.
\begin{enumerate}[(i)]
\item
If $\cP$ is hereditary, monotone, or minor closed,
then
$\chi_{\cP}(G;x)$ is an EE-invariant iff 
$\chi_{\cP}(G;x)$ 
is the chromatic polynomial $\chi(G;x)$.
\item
If $\cP$ is ultimately clique-free (biclique-free),
$\chi_{\cP}(G;x)$ is not $\MSOL_g$-definable.
\end{enumerate}
\end{theorem}

The proof of (i) appears as Theorem~\ref{th:noteei}, and
the proof of (ii) appears as Theorem~\ref{thm:12}.
\begin{remark}
If $\cP$ consists of all complete graphs or all connected graphs, $\cP$ is not ultimately clique-free,
hence Theorem \ref{th:main-1} does not apply to
the Harary polynomials $A(G;x)$ and $C(G;x)$. Nevertheless, analogue results are presented in
Sections \ref{se:EEI} and \ref{se:connection}. 
\end{remark}

\subsection{Sparsity}
For a systematic study of sparsity (and density) of graph properties 
see \cite{nevsetvril2012sparsity,nevsetvril2016structural}.

\begin{theorem}
\label{th:sparse}
\begin{enumerate}[(i)]
\item
Turan's Theorem
(\cite{turan1941extremal} and \cite[Chapter 8.3]{frieze2016introduction}):
\\
Let $G$ be $K_t$-free. Then $|E(G)| \leq (1 - \frac{1}{r})\frac{n^2}{2}$.
\item
(\cite{kovari1954problem})
Let $G$ be $K_{t,t}$-free. Then $|E(G)|  = O(n^{2 - \frac{1}{t}})$.
\item
(\cite{telle2019fpt})
If a graph property $\cP$ is 
nowhere dense or degenerate (or equivalently uniformly sparse) then $\cP$ is ultimately
biclique-free.
\item
(\cite{telle2019fpt})
There are graph properties $\cP_1, \cP_2$ which are both ultimately biclique-free but $\cP_1$ is  not degenerate 
and $\cP_2$ is not nowhere dense.
\end{enumerate}
\end{theorem}
In the light of Theorem \ref{th:sparse} ultimately
biclique-free is renamed to {\em weakly sparse}
in \cite{nesetril2019classes}. 
However, ultimately clique-free graphs can be rather dense, with $c(t) \cdot n^2$ edges rather than $n^{2 - \epsilon(t)}$ edges.

Theorem \ref{th:main-1} together with Theorem \ref{th:sparse}
shows that Harary polynomials which are EE-invariants or $\MSOL_g$-definable have to be defined using
dense properties $\cP$ as required by Tur\'an's Theorem.

\section{Graph polynomials which are not Harary polynomials}
\label{se:notharary}
Many familiar graph polynomials are not Harary polynomials of the form $\chi_{\cP}(G;x)$.
We generalize here {\cite[Theorem 5.7]{makowsky2019logician}}.
\begin{lemma}
\label{le:1}
For every graph property $\cP$ we have
$$
\chi_{\cP}(G;1) =
\begin{cases}
1 & G \in \cP, \\
0 & G \not \in \cP.
\end{cases}
$$
\end{lemma}
Using Lemma \ref{le:1} we get
\begin{proposition}
\label{pr:notharary}
Let $F(G;x)$ be a graph polynomial and $G$ be a graph such that $F(G;1) \neq 0$ and $F(G;1) \neq 1$.
Then there is no graph property $\cP$ such that $\chi_{\cP}(G;x) = F(G;x)$.
\end{proposition}

The characteristic polynomial $char(G;x)$ of a graph is the characteristic polynomial of its
adjacency matrix, and the Laplacian polynomial $Lap(G;x)$ is the characteristic polynomial of its
Laplace matrix, see \cite{bk:BrouwerHaemers2012}.

The matching polynomials are defined using $m_i(G)$, the number of matchings of $G$ of size $i$.
$$
M(G;x) = \sum_i m_i(G) x^i \mbox{ and } \mu(G;x) = \sum_i (-1)^i m_i(G) x^{n-2i}.
$$
$M(G;x)$ is the {\em generating matching polynomial} and $\mu(G;x)$ is the {\em matching defect polynomial},
see \cite{bk:LovaszPlummer86}.

Let $in_i(G)$ be the number of of independent sets of $G$ of size $i$, and
$d_i(G)$ the number of dominating sets of $G$ of size $i$.
We define the {\em independence polynomial} $IND(G;x)$, \cite{pr:LevitMandrescu05},
and the {\em domination polynomial} $DOM(G;x)$,
\cite{arocha2000mean,phd:Alikhani,ar:KotekPreenSimonTittmanTrinks2012} as
$$
IND(G;x) = \sum_i in_i(G) x^i  \mbox{ and }  DOM(G;x) = \sum_i d_i(G) x^i.
$$ 
\begin{theorem}
The following are not Harary polynomials of the form $\chi_{\cP}(G;x)$:
\begin{enumerate}[(i)]
\item
The characteristic polynomial $char(G;x)$ and the Laplacian polynomial $Lap(G;x)$.
\item
The generating matching polynomial $M(G;x)$  and the defect matching polynomial $\mu(G;x)$.
\item
The independence polynomial $IND(G;x)$.
\item
The domination polynomial $DOM(G;x)$.
\end{enumerate}
\end{theorem}
\begin{proof}
We use Proposition \ref{pr:notharary}.
(i):
$char(C_4;x) = (x-2)x^2(x+2)$ and  $Lap(C_4;x) =  x(x-4)(x-2)^2$,
\\
hence $char(C_4;1) = Lap(C_4;1) =   -3$.

(ii):
$ M(C_4;x) =4x +2x^2 \mbox{  and  }  \mu(C_4;x) = 1 +4x + 2x^2$,  
\\
hence $ M(C_4;1) =6 \mbox{  and  }  \mu(C_4;1) = 7$, 

(iii):
$ IND(C_4;x) = 1+ 4x +2x^2$, hence $IND(C_4;1) = 7$.

(iv):
$ DOM(K_2;x) =2x+x^2$, hence  $DOM(K_2;1)= 3$.
\hfill $\Box$
\end{proof}

$DOM$ and $IND$ are special cases graph polynomials of the form
$$
\cP_{\Phi}(G;x) = \sum_{A \subseteq V(G): \Phi(A)} x^{|A|}.
$$
Graph polynomials of this form are {\em generating functions} 
counting subsets $A \subseteq V(G)$ satisfying a property $\Phi(A)$,
in the cases above,
that $A$ is an independent, respectively a dominating set, see also \cite{makowsky2019logician}.
We say that {\em $\Phi$ determines $A$}, 
if for every graph $G$ there is a unique $A \subseteq V(G)$ which satisfies $\Phi(A)$.

\begin{theorem}
Assume that $\Phi$ does not determine $A$, then there is no graph property $\cP$ such that
for all graphs $G$
$\chi_{\cP}(G;x) = \cP_{\Phi}(G;x)$.
Hence $\cP_{\Phi}(G;x)$ cannot be a Harary polynomial.
\end{theorem}
\begin{proof}
By Lemma \ref{le:1} $\chi_{\cP}(G;1) \in \{0,1\}$ for all graphs $G$.
However, since $\Phi$ does not determine $A$,  there is a graph $H$ with $\cP_{\Phi}(H;1) \geq 2$.
\hfill $\Box$
\end{proof}

\section{Are Harary polynomials edge elimination invariants?}
\label{se:EEI}
\subsection{Chromatic invariants}
Following \cite[Chapter 9.1]{bk:Aigner2007},
a function $f$ which maps graphs into a polynomial ring $\mathcal{R}=\mathcal{K}[\bar{X}]$ with coefficents
in a field  $\mathcal{K}$ of characteristic $0$ is called a {\em chromatic invariant} (aka {\em Tutte-Grothendieck invariant})
if the following hold.
\begin{enumerate}[(i)]
\item
If $G$ has no edges, $f(G)=1$.
\ifskip\else
\item
Let $G$ consist of a single edge $e=(v_1,v_2)$. 
\\
If $v_1 \neq v_2$ ($e$ is a bridge), $f(G) =A \in \mathcal{R}$.
\\
If $v_1 = v_2$ ($e$ is a loop), $f(G) =B \in \mathcal{R}$.
\fi 
\item
If  $e \in E(G)$ is a bridge, then $f(G)= A \cdot f(G_{-e})$. 
\item
If  $e \in E(G)$ is a loop, then $f(G)= B \cdot f(G_{-e})$. 
\item
There exist $\alpha, \beta \in \mathcal{R}$ such that
for every $e \in E(G)$ which is neither a loop nor a bridge we have
$
f(G) = \alpha \cdot f(G_{-e}) + \beta \cdot f(G_{/e})
$.
\item
Multiplicativity: 
If $G = G_1 \sqcup G_2$ is the disjoint union of two graphs $G_1, G_2$ then
$
f(G) = f(G_1) \cdot f(G_2)
$.
\end{enumerate}

Chromatic invariants have a characterization via the Tutte polynomial $T(G;x,y)$, see
\cite[Chapter 9.1, Theorem 9.5]{bk:Aigner2007}.

\begin{theorem}
Let $f$ be a chromatic invariant with $A, B, \alpha, \beta$ indeterminates as above. Then for all graphs
$G$ 
$$
f(G) =  
\alpha^{|E| - |V| + k(G)} \cdot \beta^{|V| -k(G)} \cdot T(G; \frac{A}{\beta}, \frac{B}{\alpha}).
$$
\end{theorem}
It follows by a counting argument that not all Harary polynomials are chromatic invariants.
We characterize the
Harary polynomials which are chromatic invariants
in Theorem \ref{th:noteei} below.

\subsection{Edge elimination invariants}
\label{sec:Xi}

The Tutte polynomial generalizes the chromatic, flow and other graph polynomials. 
It is natural to search for polynomials that generalize it, in turn. 
The {\em Most General Edge Elimination Invariant}, introduced in \cite{averbouch2010extension},\cite{averbouch2008most} 
and also known as the {\em $\xi$ polynomial}, generalizes the Tutte and the matching polynomials.
\begin{definition}[Edge Elimination Invariant]
Let $F$ be a graph parameter with values in a ring $\cR$.
$F$ is an {\em EE-invariant} if there exist $\alpha, \beta, \gamma \in \cR$
such that
\begin{align}
F(G)=F(G_{-e})+\alpha F(G_{/e})+ \beta F(G_{\dagger e})
\end{align}
where $e\in E(G)$, with the additional conditions
\begin{align}
F(\emptyset)=1, \hspace{0.3cm}
F(K_1)=\gamma, \ \ \  \text{and}  \ \ 
F(G \sqcup H)=F(G) \cdot F(H).
\end{align}
\end{definition}

Let $\xi(G;x,y,z)$ be the graph polynomial 
$$
\xi(G;x,y,z)=\sum_{A,B\subseteq E(G)} x^{c(A\cup B)-cov(B)}y^{|A|+|B|-cov(B)}z^{cov(B)},
$$
where the summation is over $A,B\subseteq E(G)$ such that the vertex subsets $V(A), 
V(B)$ covered by $A$ and $B$, respectively, are disjoint, 
$c(A)$ is the number of connected components in $(V(G),A)$, 
and $cov(B)$ is the number of covered connected component of $B$, 
i.e. the number of connected components of $(V(B),B)$.

\begin{theorem}[\cite{averbouch2008most}]
\begin{enumerate}[(i)]
\item
$\xi(G;x,y,z)$ 
is an EE-invariant.
\item
Every EE-invariant is a substitution instance of $\xi(G;x,y,z)$
multiplied by some factor $s(G)$ which only depends on the number of vertices, edges and connected components of $G$.
\item
Both the matching polynomial and the Tutte polynomial are EE-invariants given by
$$
T(G;x,y)=(x-1)^{-c(E(G))}(y-1)^{-|V(G)|}\xi(G;(x-1)(y-1),y-1,0),
$$
and
$$
M(G;w_1,w_2)=\xi(G;w_1,0,w_2).
$$
\end{enumerate}
\end{theorem}

\subsection{Are Harary polynomials EE-invariants?}

\ifskip
\else
\begin{lemma}
\label{le:ee}
Let $\cP$ be monotone (hereditary,  minor closed) and $H$ be a graph of smallest order, 
and among those of smallest and size, such that
$H$ is a forbidden subgraph (induced subgraph, minor) of $\cP$.
Assume further that $H$
has at least four vertices and one edge $e \in E(H)$.
Then $\chi_{\cP}(G;x)$ 
is not an EE-invariant.
\end{lemma}
\begin{proof}
As $\cP$
is monotone (hereditary, minor closed) 
we note that deleting or contracting or extracting $e$ from $H$, we obtain a graph in $\cP$.
Hence we can compute:
\begin{gather}
\chi_{\cP}(H,x) = x^{|V(H)|} -x \\
\chi_{\cP}(H_{-e},x) = x^{|V(H)|} \\
\chi_{\cP}(H_{/e},x) = x^{|V(H)|-1} \\
\chi_{\cP}(H_{\dagger e},x) = x^{|V(H)|-2} 
\end{gather}

Now assume that $\chi_{\cP}$ is a EE-invariant.
Then we have
\begin{gather}
\chi_{\cP}(H,x) = x^{|V(H)|} -x \tag{*} \\
=
\chi_{\cP}(H_{-e},x) +
\alpha(x) \cdot \chi_{\cP} (H_{/e},x) +
\beta(x) \cdot \chi_{\cP}(H_{\dagger e},x) \notag \\
=
x^{V(H)} +
\alpha(x) \cdot x^{V(H)-1} +
\beta(x) \cdot x^{V(H)-2}  \tag{**}
\end{gather}
for $\alpha(x), \beta(x) \in \ZZ[x]$ polynomials in $x$.

If $|V(H)| \geq 4$ 
the coefficient of $x$ in (*) is $-1$, and in (**) it is $0$ which is a contradiction.
\hfill $\Box$
\end{proof}

\begin{theorem}
\label{th:noteei}
Assume that $\cP$ is monotone (hereditary, minor closed) and ultimately clique-free, but contains no edgeless graph. Then 
$\chi_{\cP}(G;x)$ is not an EE-invariant.
\end{theorem}
\begin{proof}
For every $G \in \cP$ the set of edges $E(G) \neq \emptyset$,  $K_3$ is a forbidden subgraph (induced subgraph, minor).
If $K_4$ is a forbidden subgraph (induced subgraph, minor) of $\cP$, we apply Lemma  \ref{le:ee}.
Otherwise we compute
$\chi_{\cP}(K_3,x) = x^3-x$.
Let $H= K_1 \bowtie (K_2 \sqcup K_1)$.
$\chi_{\cP}(H, x)$ has degree $4$, since $\chi_{\cP}(H, x)$ is monic and $H$ has order $4$ (Facts \ref{facts}).
\\
Assuming that $\chi_{\cP}(H, x)$ is an EE-invariant,
$\chi_{\cP}(K_3,x) =x^3-x$ gives $\alpha=0$ and $\beta=-1$.
Computing $\chi_{\cP}(H,x)$ we get 
$
\chi_{\cP}(H,x) = \chi_{\cP}(K_3) + \alpha \chi_{\cP}(K_3) + \beta \chi_{\cP}(K_2) 
$.
But this is 
a polynomial of degree $3$, which is a contradiction.
\hfill $\Box$
\end{proof}
\fi 
\begin{theorem}
\label{th:noteei}
Let $\cP$ be a non-trivial (minor closed/monotone/hereditary) graph property. 
Then $\chi_\cP$ is an EE-invariant if and only if $\chi_\cP$ is the chromatic polynomial.
\end{theorem}

We need a lemma:
\begin{lemma}
\label{le:multiplicative}
Let $\cP$ be a non-trivial (minor closed/monotone/hereditary) graph property, and
$H$ a forbidden  minor, subgraph or induced subgraph of $\cP$.
\\
Assume $H = H_1 \sqcup H_2$ with both $H_1$ and $H_2$ in $\cP$. Then $\chi_\cP$  is not multiplicative.
\end{lemma}

\begin{corollary}
In particular,  if $H=E_n$ or $H=K_1 \cup K_2$, $\chi_\cP$ is not multiplicative. 
\end{corollary}
\begin{proof}
\black
If $H=H_1 \sqcup H_2$,  both $H_1, H_2$ are
minors, subgraphs, and induced subgraphs.
Hence we have 
$$
0=\chi_\cP(H,1)=\chi_\cP(H_1 \sqcup H_2,1) 
$$
as $H$ is forbidden, but
$$
\chi_\cP(H_1,1) \cdot\chi_\cP(H_2,1)=1.
$$
\end{proof}

\begin{proof}[of Theorem \ref{th:noteei}]
We analyze $H$,
a forbidden 
(minor/subgraph/induced subgraph) of $\cP$ with the smallest number of vertices and edges.
\\
The proof distinguishes between cases:
\begin{enumerate}[(i)]
\item
$H$ is not connected.
\item
$H=K_1$.
\item
$H = K_2$. 
\item
$H = P_3$
\item
$H = K_3$
\item
$H$ has order $\geq 4$ and size $\geq 1$.
\end{enumerate}
Case (i): $H$ is not connected.
\\
We use Lemma \ref{le:multiplicative}.

Case (ii): $H = K_1$  
\\
If $H =K_1$, $\cP$ is empty, hence trivial.

Case (iii): $H = K_2$  
\\
If $H = P_2$, then $\chi_\cP$  is the chromatic polynomial.

Case (iv) $H = P_3$
\\ 
We compute:
$ \chi_\cP(K_1,x)=x, \ \ \ \chi_\cP(K_2,x)=x^2$ and
\begin{gather}
\chi_\cP(P_3,x)=x^3-x, \ \ \ \chi_\cP(K_1\cup K_2,x)=x^3
\tag{*}
\end{gather}
Assuming that $\chi_\cP$ is an EE-invariant, we can apply the recursive relation to get:
\begin{gather}
\chi_\cP(P_3,x)=\chi_\cP(K_1 \cup K_2,x)+\alpha(x)\chi_\cP(K_2,x)+\beta(x)\chi_\cP(K_1,x) \notag \\
=x^3+\alpha(x)x^2+\beta(x)x \tag{**}\\
\chi_\cP(K_1 \cup K_2,x)=\chi_\cP(E_3,x)+\alpha(x)\chi_\cP(E_2,x)+\beta(x)\chi_\cP(K_1,x) \notag \\
=x^3+\alpha(x)x^2+\beta(x)x \tag{***}
\end{gather}
By combining (*) with (**) and (***)
$$
{\red  -x} =\alpha(x)x^2+\beta(x)x =
{\red 0}
$$
which is a contradiction.

Case (v) $H = K_3$.
\\
We compute:
$ \chi_\cP(K_1,x)=x, \ \ \ \chi_\cP(K_2,x)=x^2$ and
\begin{gather}
\chi_\cP(K_3,x)=x^3-x, \ \ \ \chi_\cP(K_1\cup K_2,x)=x^3
\tag{+}
\end{gather}
Assuming that $\chi_\cP$ is an EE-invariant, we can apply the recursive relation to get:
\begin{gather}
\chi_\cP(K_3,x)=\chi_\cP(K_2,x)+\alpha(x)\chi_\cP(K_2,x)+\beta(x)\chi_\cP(K_1,x)\notag \\
=x^3+\alpha(x)x^2+\beta(x)x \tag{++} \\
\chi_\cP(P_3,x)=\chi_\cP(K_1 \cup K_2,x)+\alpha(x)\chi_\cP(K_2,x)+\beta(x)\chi_\cP(K_1,x) \notag \\
=x^3+\alpha(x)x^2+\beta(x)x \tag{+++}
\end{gather}
By combining (+) with (++) and (+++)
$$
{\red  -x} =\alpha(x)x^2+\beta(x)x =
{\red 0}
$$
which is a contradiction.

Case (vi): $H$ has order $\geq 4$ and size $\geq 1$.
\\
We note that deleting or contracting or extracting $e$ from $H$, we obtain a graph in $\cP$.
Hence we can compute:
\begin{gather}
\chi_{\cP}(H,x) = x^{|V(H)|} -x \text{   and   }
\chi_{\cP}(H_{-e},x) = x^{|V(H)|} \notag \\
\chi_{\cP}(H_{/e},x) = x^{|V(H)|-1} \notag  \text{   and   }
\chi_{\cP}(H_{\dagger e},x) = x^{|V(H)|-2}\notag  
\end{gather}
Now we assume that $\chi_{\cP}$ is an EE-invariant and get
\begin{gather}
\chi_{\cP}(H,x) = x^{|V(H)|} -x \tag{*} \\
=
\chi_{\cP}(H_{-e},x) +
\alpha(x) \cdot \chi_{\cP} (H_{/e},x) +
\beta(x) \cdot \chi_{\cP}(H_{\dagger e},x) \notag \\
=
x^{V(H)} +
\alpha(x) \cdot x^{V(H)-1} +
\beta(x) \cdot x^{V(H)-2}  \tag{**}
\end{gather}
for $\alpha(x), \beta(x) \in \ZZ[x]$ polynomials in $x$.
\\
If $|V(H)| \geq 4$ the coefficient of $x$ in (*) is $\red -1$, \\
and in (**) it is $\red 0$, which is a contradiction.
\hfill $\Box$
\end{proof}

The graph polynomials $C(G;x)$ and $A(G;x)$ are Harary polynomials where the property $\cP$ 
contains arbitrarily large cliques.

\begin{proposition}
Both $C(G;x)$ and $A(G;x)$ are not multiplicative, hence they are not EE-invariants.
\end{proposition}
\begin{proof}
$C(K_1,x) = A(K_1,x)=x$ and 
$C(K_1 \sqcup K_1,x)= A(K_1 \sqcup K_1,x)=x^2-x \neq x^2$.
\hfill $\Box$
\end{proof}

\section{$\MSOL$-definable graph polynomials}
\label{se:sol}
We assume the reader is familiar with Second and Monadic Second Order Logic
for graphs. A good source is \cite{courcelle2012graph,makowsky2014connection,kotek2016logic,makowsky2019logician}.
We distinguish between $\MSOL$ for the language of graphs, with one binary edge relation $\MSOL_g$,
and $\MSOL$ for  the language of hypergraphs $\MSOL_h$, 
with vertices and edges as elements and a binary incidence relation.
We also refer to second order logic $\SOL_g, \SOL_h$ in a similar way.

A {\em simple univariate 
$\MSOL_g$-definable 
($\MSOL_h$, $\SOL_g$, $\SOL_h$-definable) 
graph polynomial $F(G;x)$}
is a polynomial of the form
$$
F(G;x) = \sum_{A \subseteq V(G): \phi(A)} \prod_{v \in I} x,
$$
where $A$ ranges over all subsets of $V(G)$ satisfying $\phi(A)$
and $\phi(A)$ is a $\MSOL_g$-formula.
$F$ is $\MSOL_h$-definable if $A$ ranges over $V(G) \cup E(G)$
and $\phi(A)$ is a $\MSOL_h$-formula.
$F$ is $\SOL_g$-definable if $A$ ranges over $V(G)^m$.
$F$ is $\SOL_h$-definable if $A$ ranges over $(V(G) \cup E(G))^m$.

\begin{examples}
\begin{enumerate}[(i)]
\item
The independence polynomial
$IND(G;x) =\sum_i ind(G,i) \cdot x^i$,
can be written as
$$
IND(G,x) = \sum_{I \subseteq V(G)} \prod_{v \in I} x,
$$
where $I$ ranges over all independent sets of $G$.
To be an independent set is $\MSOL_g$-definable.
\item
The matching generating polynomial $M(G;x)$ is $\MSOL_h$-definable, but unlikely to be $\MSOL_g$-definable,
\cite{makowsky2006computing}, otherwise it would be fixed parameter tractable for clique-width at most $k$.
\end{enumerate}
\end{examples}

For the general case
one allows several indeterminates $x_1, \ldots, x_m$, and gives 
an inductive definition.
One may also allow an ordering of the vertices, but then one requires
the definition to be
{\em invariant under the ordering}, 
i.e., different orderings still give the same polynomial.

\begin{examples}
The Tutte polynomial is a bivariate $\MSOL_h$-definable graph polynomial using an ordering on the vertices,
\cite{makowsky2005coloured}.
Similarly, it can be shown that 
the polynomial $\xi(G;x,y,z)$ is a  trivariate $\MSOL_h$-definable graph polynomial
using an ordering on the vertices, \cite{averbouch2008most}. For a proof see Theorem \ref{th:xi-2}.
\end{examples}

A univariate graph polynomial is 
$\MSOL_g$-definable 
($\MSOL_h$, $\SOL_g$, $\SOL_h$-definable) 
if it is a substitution instance of a multivariate
$\MSOL_g$-definable 
($\MSOL_h$, $\SOL_g$, $\SOL_h$-definable) 
graph polynomial.

All we can say about the definability of Harary polynomials is the following:
\begin{proposition}
If $\cP$ is 
$\SOL_g$-definable 
so is the Harary polynomial $\chi_{\cP}(G;x)$.
\end{proposition}
\begin{proof}
We only prove the case where
$\cP$ is 
$\SOL_g$-definable, the other cases are similar.

Let $\phi$ be the $\MSOL_g$-formula which defines $\cP$.
Let $\Phi(X, E)$ be the formula which says that $X \subseteq V(G)^2$ is an equivalence relation on $V(G)$
such that each equivalence class induces a graph satisfying $\phi$.
Now we can write
$$
\chi_{\cP}(G;x)= \sum_{X \subseteq V(G)^2: \Phi(X, E)} x^{|X|}.
$$
\hfill $\Box$
\end{proof}

The chromatic polynomial is not $\MSOL_g$-definable, 
In the sequel we show that many Harary polynomials are not $\MSOL_g$-definable.

\section{Connection Matrices}
\label{se:connection}
In this section we  prove for many Harary polynomials
that they are not $\MSOL_g$-polynomials.

We use tools from \cite{makowsky2014connection}.
Let $G_i$ be an enumeration of all finite graphs (up to isomorphisms).
We denote by 
$G_i \sqcup G_j$ the disjoint union,
and by 
$G_i \bowtie G_j$ the join
of
$G_i$ and $G_j$.

Let $F =F(G;x) \in \ZZ[x]$ be a graph polynomial.
Let $\mathcal{H}(\bowtie, F)$ be the infinite matrix where rows and columns are labeled
by $G_i$. Then we define 
\begin{gather}
\mathcal{H}(\bowtie, F) (G_i, G_j) = F(G_i \bowtie G_j;x),
\notag \\
\mathcal{H}(\sqcup, F) (G_i, G_j) = F(G_i \sqcup G_j;x).
\notag
\end{gather}
$\mathcal{H}(\bowtie, F)$,
respectively
$\mathcal{H}(\sqcup, F)$, 
is called a {\em connection matrix} aka {\em Hankel matrix}.

\begin{theorem}[\cite{makowsky2014connection}]
\label{th:KM}
If $F(G;x)$ is $\MSOL_g$-definable, then
$\mathcal{H}(\bowtie, F)$ and
$\mathcal{H}(\sqcup, F)$
have finite rank over the ring $\ZZ[x]$.
\end{theorem}

The following lemmas are needed for Theorem \ref{thm:12}.
\begin{lemma}
\label{lm:non-def-gp}
Given a  graph polynomial $F$, 
and an infinite sequence of non-isomorphic graphs
$H_i, i \in \NN$,
let $f:\NN\to \NN$ be an unbounded function such that 
for every $k \in \NN$, 
$F(H_i\bowtie H_j,k)=0$ iff $i+j>f(k)$.
\\
Then the matrix $\mathcal{H}(\bowtie,p)$ has infinite rank.
\\
The same also holds when $\bowtie$ is replaced by the disjoint union $\sqcup$.
\end{lemma}

Given a graph $H$
we denote by $Forb^{sub}(H)$ ($Forb^{ind}(H)$) the class of graphs
which do not contain an (induced) subgraph isomorphic to $H$.
If $H$ is a complete graph the two classes coincide, and we omit the superscript.

We now prove specific cases where we can apply Lemma \ref{lm:non-def-gp} with
$H_i = K_i$ the complete graph on $i$ vertices. 
\begin{lemma}
\label{le:forb(kh)}
\label{le:forb(h)}
\begin{enumerate}[(i)]
\item
Let $\cP_1 \subseteq Forb(K_h)$.
Then $\chi_{\cP_1}(K_i;k) =0$  iff $i > hk$.
\item
Let $\cP_2 \subseteq Forb^{sub}(H)$ for some
connected graph $H$ on $h$ vertices. 
Then $\chi_{\cP_2}(K_i;k) =0$  iff $i > hk$.
\end{enumerate}
\end{lemma}
\begin{proof}
If we partition a set of size $i > hk$ into $k$ disjoint sets, at least on of these sets has size $>h$.
Hence, if we partion $K_i$, at least one of these sets induces a $K_h$. So hence $\chi_{\cP_1}(K_i;k) =0$.
Since $H$ is a subgraph of $K_h$, $\chi_{\cP_2}(K_i;k) =0$.
\\
Note that for $\cP = Forb(K_2)$ this is the chromatic polynomial.
\hfill $\Box$
\end{proof}


\begin{theorem} \label{thm:12}
Let $\cP$ be a non-trivial graph property.
If $\cP$ is (i) monotone, (ii) ultimately clique-free, or (iii) weakly sparse,
the Harary polynomial $\chi_{\cP}(G;x)$ is not $\MSOL_g$-definable.
\end{theorem}
\begin{proof}
(i): 
If $P$ is non-trivial and monotone there is a connected graph $H$ with $P \subset Forb^{sub}(H)$.
By Lemma \ref{le:forb(h)}(ii) we get 
$\chi_P(K_i;k) =0$  iff $i > hk$.
\\
By Lemma \ref{lm:non-def-gp}, $\mathcal{H}(\bowtie,\chi_P(G;x))$
has infinite rank. Now we use Theorem \ref{th:KM}.
\\
(ii):
$\cP \subseteq Forb(K_h)$, hence
by Lemma \ref{le:forb(kh)}(i), we get again
$\chi_P(K_i;k) =0$  iff $i > hk$. Then we proceed as in (i).
\\
(iii):
$\cP$ is ultimately clique-free hence 
there is $h \in \NN$ with 
$P \subseteq Forb(K_h)$. So we proceed as in (ii).
\hfill $\Box$
\end{proof}

The graph polynomials $C(G;x)$ and $A(G;x)$ are Harary polynomials where the property $\cP$
contains graphs of maximal density.

\begin{proposition}
Both $C(G;x)$ and $A(G;x)$ are not $\MSOL_g$-definable.
\end{proposition}
\begin{proof}
In both cases we look at the graph $M_n$ of order $2n$ which consists of $n$ disjoint copies
of $K_2$. We note that $M_i \sqcup M_j = M_{i+j}$, and we get
$C(M_n;k)= A(M_n;k)= 0$ for  $ n > 2k$. So we can apply Lemma \ref{lm:non-def-gp}
with the join replaced by the disjoint union.
\hfill $\Box$
\end{proof}

\section{Conclusions and Open Problems}
\label{se:conclu}
We have initiated a systematic study of Harary polynomials.

In this paper
we have shown that among 
the  Harary polynomials $\chi_{\cP}$ with $\cP$
hereditary monotone, or minor closed,
the chromatic polynomial is the only EE-invariant.

We have also shown that the  Harary polynomials $\chi_{\cP}$
are not $\MSOL_g$-definable if $\cP$ is either monotone, ultimately clique-free, or weakly sparse.
This includes the chromatic polynomial.
However, the chromatic polynomial is $\MSOL_h$-definable.

\begin{question}
Is there a Harary polynomial, different from the chromatic polynomial, which is
$\MSOL_h$-definable 
and/or an EE-invariant?
\end{question}

We suspect (but do not conjecture) that
the chromatic polynomial is the 
only Harary polynomial which is an EE-invariant? 

In future research we continue the study 
of the complexity of evaluating Harary polynomials,
initiated in \cite{goodall2018complexity}. 

\paragraph{ \bf Acknowledgement:} We would like to thank P. Tittmann for his careful reading of drafts of this paper.
We would also like to thank three anonymous referees of the previous version of this paper for
their valuable and justified comments.
This research profited also a lot from two Dagstuhl Seminars:
\begin{itemize}
\item
Dagstuhl Seminar 16241, June 12--17, 2016, 
\\
Graph Polynomials: Towards a Comparative Theory
\item
Dagstuhl Seminar 19401, September 29--October 4, 2019, 
\\
Comparative Theory for Graph Polynomials
\end{itemize}


\bibliographystyle{splncs04}
\bibliography{wg-ref,MM-ref}
\appendix
\section{Proof of Theorem \ref{th:xi-2}}
\label{app:xidef}
Theorem
\ref{th:xi-2}
states that
the most general EE-invariant $\xi(G;x,y,z)$ is $\MSOL_h$-definable for graphs with a linear order
on the vertices. Furthermore, this definition is invariant under the particular order of the vertices.
\begin{proof}
\begin{gather}
\xi(G;x,y,z)=\sum_{A,B\subseteq E(G)} x^{c(A\cup B)-c_{cov}(B)}y^{|A|+|B|-c_{cov}(B)}z^{c_{cov}(B)}=
\notag \\
\sum_{A,B\subseteq E(G)} x^{c(A\cup B)}y^{|A|+|B|} (\frac{z}{xy})^{c_{cov}(B)}
\notag
\end{gather}

Let $(V(G), E(G), Ord(G))$ be a graph with a linear ordering of the vertices.
\begin{itemize}
\item
Let $\phi(A,B)$ be the formula $cov(A) \cap cov(B) = \emptyset$ with
\\
$cov(A,v) =  \exists  v_1, v_2 ( (v_1, v_2) \in A \wedge (v = v_1 \vee v = v_2))$ 
where $A$ ranges over sets of edges, and $v, v_1, v_2$ range over vertices.
\item
We write 
for sets of edges $A, B$:
\begin{gather}
X^{|A|} = \prod_{e: e \in A} X 
\text{   and   }
X^{|A|+|B|} = \prod_{e: e \in A \sqcup B} X 
\notag
\end{gather}
where
$A \sqcup B$ is the disjoint union of $A$ and $B$.
\item
We write for a vertex $v$ and a set of edges $A$:
\begin{gather}
X^{c(F)} = \prod_{v: \phi_c(F,v)} X 
\notag
\end{gather}
where $\phi_c(F, v)$ says that $v$ is the first vertex of a connected component of the graph $(V(G),F)$. 
If $F = A \sqcup B$ we use instead
$\psi_c(A, B, v)$ which says $ \exists F (\phi_c(F, v) \wedge F= A \sqcup B)$.
\end{itemize}

\begin{itemize}
\item
We write for a vertex $w$ and a set of edges $B$:
\begin{gather}
X^{cov(A)} = \prod_{w: \phi_{cov}(B,w)} X
\notag
\end{gather}
where $\phi_{cov}(B,w)$ says that $w$ is the first vertex of a connected component of the graph $(V(B),B)$. 
\end{itemize}
Then we can write
$$
\xi(G;x,y,z)=
\sum_{A, B: \phi(A,B) }
\left(
\prod_{u: \phi_c(A,B,u)} x
\cdot
\prod_{e: e \in A \sqcup B} y
\cdot
\prod_{w: \phi_{cov}(B,w)} \frac{z}{xy}
\right)
$$

\hfill $\Box$
\end{proof}

\end{document}